\documentclass[a4paper, 11pt]{amsart}
\usepackage[utf8]{inputenc}
\usepackage[english]{babel}
\usepackage[T1]{fontenc}
\usepackage{amsmath, amssymb, amsfonts}
\usepackage[all]{xy}
\usepackage{enumerate}
\usepackage{graphicx}
\usepackage[hmargin=1.5cm, vmargin=2.5cm]{geometry}
\usepackage{rotate}
\usepackage{MnSymbol}
\usepackage{comment}
\usepackage{tikz}
\usepackage{hyperref}

\title{On bi-free De Finetti theorems}
\author{Amaury Freslon}
\author{Moritz Weber}
\address{Saarland University, Fachbereich Mathematik, Postfach 151159, 66041 Saarbr\"ucken, Germany}
\email{freslon@math.uni-sb.de, weber@math.uni-sb.de}
\date{}

\theoremstyle{plain}
\newtheorem{thm}{Theorem}[section]
\newtheorem{prop}[thm]{Proposition}

\newtheorem{lem}[thm]{Lemma}

\theoremstyle{definition}
\newtheorem{de}[thm]{Definition}
\newtheorem{ex}[thm]{Example}

\theoremstyle{remark}
\newtheorem{rem}[thm]{Remark}

\DeclareMathOperator{\E}{E}
\DeclareMathOperator{\Id}{Id}
\DeclareMathOperator{\ii}{id}

\newcommand{\A}{\mathcal{A}}
\newcommand{\B}{\mathcal{B}}
\newcommand{\C}{\mathbb{C}}
\newcommand{\CC}{\mathcal{C}}

\newcommand{\D}{\Delta}

\newcommand{\EE}{\mathbb{E}}

\newcommand{\G}{\mathbb{G}}

\newcommand{\LL}{\mathcal{L}}
\newcommand{\M}{\mathcal{M}}
\newcommand{\N}{\mathbb{N}}
\newcommand{\NN}{\mathcal{N}}

\begin{document}

\setlength{\unitlength}{0.5cm}
\newsavebox{\partpi}
\savebox{\partpi}
{ \begin{picture}(9,3)
 \put(0,2){\line(0,1){1}}
 \put(1,2){\line(0,1){1}}
 \put(2,2){\line(0,1){1}}
 \put(3,0){\line(0,1){3}}
 \put(4,2){\line(0,1){1}}
 \put(5,1){\line(0,1){2}}
 \put(6,0){\line(0,1){3}}
 \put(7,1){\line(0,1){2}}
 \put(8,2){\line(0,1){1}}
 \put(0,2){\line(1,0){1}}
 \put(2,2){\line(1,0){6}}
 \put(3,0){\line(1,0){3}}
 \put(5,1){\line(1,0){2}}
 \put(-0.1,3.2){$\ell$}
 \put(0.9,3.2){$r$}
 \put(1.9,3.2){$\ell$}
 \put(2.9,3.2){$\ell$}
 \put(3.9,3.2){$r$}
 \put(4.9,3.2){$r$}
 \put(5.9,3.2){$\ell$}
 \put(6.9,3.2){$r$}
 \put(7.9,3.2){$r$}
 \end{picture}}
 \newsavebox{\partpinumber}
\savebox{\partpinumber}
{ \begin{picture}(9,3)
 \put(0,2){\line(0,1){1}}
 \put(1,2){\line(0,1){1}}
 \put(2,2){\line(0,1){1}}
 \put(3,0){\line(0,1){3}}
 \put(4,2){\line(0,1){1}}
 \put(5,1){\line(0,1){2}}
 \put(6,0){\line(0,1){3}}
 \put(7,1){\line(0,1){2}}
 \put(8,2){\line(0,1){1}}
 \put(0,2){\line(1,0){1}}
 \put(2,2){\line(1,0){6}}
 \put(3,0){\line(1,0){3}}
 \put(5,1){\line(1,0){2}}
 \put(-0.1,3.2){$1$}
 \put(0.9,3.2){$2$}
 \put(1.9,3.2){$3$}
 \put(2.9,3.2){$4$}
 \put(3.9,3.2){$5$}
 \put(4.9,3.2){$6$}
 \put(5.9,3.2){$7$}
 \put(6.9,3.2){$8$}
 \put(7.9,3.2){$9$}
 \end{picture}}
\newsavebox{\partpullleftright}
\savebox{\partpullleftright}
{ \begin{picture}(9,2)
 \put(0,0){\line(0,1){2}}
 \put(1,0){\line(3,2){3}}
 \put(2,0){\line(-1,2){1}}
 \put(3,0){\line(-1,2){1}}
 \put(4,0){\line(1,2){1}}
 \put(5,0){\line(1,2){1}}
 \put(6,0){\line(-3,2){3}}
 \put(7,0){\line(0,1){2}}
 \put(8,0){\line(0,1){2}}
 \end{picture}}
\newsavebox{\partreverseorder}
\savebox{\partreverseorder}
{ \begin{picture}(9,3)
 \put(0,0){\line(0,1){3}}
 \put(1,0){\line(0,1){3}}                        
 \put(2,0){\line(0,1){3}}                        
 \put(3,0){\line(0,1){3}}                        
 \put(4,0){\line(4,3){4}}
 \put(5,0){\line(2,3){2}}
 \put(6,0){\line(0,1){3}}
 \put(7,0){\line(-2,3){2}}
 \put(8,0){\line(-4,3){4}} 
 \put(-0.1,3.2){1}
 \put(0.9,3.2){2}
 \put(1.9,3.2){3}
 \put(2.9,3.2){4}
 \put(3.9,3.2){5}
 \put(4.9,3.2){6}
 \put(5.9,3.2){7}
 \put(6.9,3.2){8}
 \put(7.9,3.2){9}
 \end{picture}}
\newsavebox{\partpermutedpi}
\savebox{\partpermutedpi}
{ \begin{picture}(9,3)
 \put(0,0){\line(0,1){3}}
 \put(1,1){\line(0,1){2}}
 \put(2,2){\line(0,1){1}}
 \put(3,2){\line(0,1){1}}
 \put(4,1){\line(0,1){2}}
 \put(5,2){\line(0,1){1}}
 \put(6,2){\line(0,1){1}}
 \put(7,1){\line(0,1){2}}
 \put(8,0){\line(0,1){3}}
 \put(0,0){\line(1,0){8}}
 \put(1,1){\line(1,0){6}}
 \put(2,2){\line(1,0){1}}
 \put(5,2){\line(1,0){1}}
 \put(-0.1,3.2){1}
 \put(0.9,3.2){2}
 \put(1.9,3.2){3}
 \put(2.9,3.2){4}
 \put(3.9,3.2){5}
 \put(4.9,3.2){6}
 \put(5.9,3.2){7}
 \put(6.9,3.2){8}
 \put(7.9,3.2){9}
 \end{picture}}

\begin{abstract}
We investigate possible generalizations of the de Finetti theorem to bi-free probability. We first introduce a twisted action of the quantum permutation groups corresponding to the combinatorics of bi-freeness. We then study properties of families of pairs of variables which are invariant under this action, both in the bi-noncommutative setting and in the usual noncommutative setting. We do not have a completely satisfying analogue of the de Finetti theorem, but we have partial results leading the way. We end with suggestions concerning the symmetries of a potential notion of $n$-freeness.
\end{abstract}

\maketitle

\section{Introduction}

D.V. Voiculescu defined in \cite{voiculescu2014free} a notion of freeness for pairs of noncommutative random variables, called \emph{bi-freeness}. Consider pairs of operators $(T_{i}^{\ell}, T_{i}^{r})_{i}$ both acting on a pointed Hilbert space $(H_{i}, \xi_{i})$. If $H = \ast_{i}H_{i}$ is the reduced (with respect to the vectors $\xi_{i}$) free product Hilbert space, then $\B(H_{i})$ can be represented on $H$ by letting the operators act on the \emph{leftmost} tensor if it is in $H_{i}$. Similarly, one can represent $\B(H_{i})$ on $H$ by letting the operators act on the \emph{rightmost} tensor if it is in $H_{i}$. If $\lambda_{i}$ (resp. $\rho_{i}$) denote this left (resp. right) representation, we can then consider the joint distribution of the family of operators $(\lambda_{i}(T_{i}^{\ell}), \rho_{i}(T_{i}^{r}))_{i}$ with respect to the vacuum expectation. If $i\neq j$, the operators $\lambda_{i}(T_{i}^{\ell})$ and $\rho_{j}(T_{j}^{r})$ are classically independent, while all those acting on the same side of $H$ but on different Hilbert spaces are free. A family of pairs is said to be bi-free if its joint distribution can be modelled in this way.

Building on methods from free probability, D.V. Voiculescu proved in \cite{voiculescu2014free} and \cite{voiculescu2013free} several fundamental results for "bi-free probability" and in particular a central limit theorem. However, he had no combinatorial description of bi-freeness and some of the constructions (like the universal polynomials for moments) were not explicit. In \cite{mastnak2014double}, M. Mastnak and A. Nica introduced combinatorial objects called \emph{bi-noncrossing partitions}. They associated a family of \emph{$(\ell, r)$-cumulants} to them and conjectured that bi-freeness was equivalent to the vanishing of these mixed cumulants. This conjecture was later proved by I. Charlesworth, B. Nelson and P. Skoufranis in \cite{charlesworth2014two}. Afterwards, the same authors endeavoured to study the operator-valued setting for bi-freeness in \cite{charlesworth2014combinatorics}. This work is highly technical, but they were able to generalize many basic results from operator-valued free probability theory to the setting of pairs of random variables. With this whole framework available, it is natural to start investigating further topics in bi-free probability.

One possible direction is to study quantum symmetries and in particular generalizations of the de Finetti theorem. It was proved by C. Köstler and R. Speicher in \cite{kostler2009noncommutative} that an infinite family of noncommutative random variables is free and identically distributed with amalgamation over a subalgebra if and only if it is invariant under a natural action of the quantum permutation group. A version for boolean independence was also developed by W. Liu in \cite{liu2014noncommutative}, involving a \emph{quantum semigroup} generalizing the quantum permutation group.

In the present paper, we study possible generalizations of the de Finetti theorem to the setting of bi-freeness. Like for freeness, the role of the quantum symmetries is played by the quantum permutation group, but its action is changed. It is twisted so that it matches the specific combinatorics of bi-noncrossing partitions. This yields quantum symmetries in the sense that a family of pairs which is bi-free and identically distributed with amalgamation is invariant under the twisted action. A de Finetti theorem should then be a converse to this statement and we will therefore study the consequences of invariance under the twisted action of the quantum permutation group. We will see however that it is quite unclear what the statement of a bi-free de Finetti theorem should be.

Let us briefly outline the content of this paper. In Section \ref{sec:preliminaries} we recall necessary background concerning bi-freeness with amalgamation and the quantum permutation group. Section \ref{sec:DeFinetti} is divided into three parts. We first introduce in Subsection \ref{subsec:linearaction} the twisted action of the quantum permutation group and prove the easy way of the de Finetti theorem. Then, we prove and discuss in Subsection \ref{subsec:weakdefinetti} a weak converse of this statement in the setting of operator-valued bi-freeness. Eventually, we address in Subsection \ref{subsec:quantumbiinvariant} the general problem of characterizing infinite families of pairs which are quantum bi-invariant. In particular, we will detail the main difficulty arising there, that is producing a $B$-$B$-noncommutative probability space out of a $B$-noncommutative probability space. In the last short Section \ref{sec:nfreeness}, we explain how quantum bi-exchangeability can naturally be extended to quantum symmetries of $n$-tuples of operators. This gives some clues on what the combinatorics of $n$-freeness should be, if such a notion exists.

\subsection*{Acknowledgements}

The first author is supported by the ERC advanced grant "Noncommutative distributions in free probability". The authors thank W. Liu for pointing out a gap in the first version of this paper and K. Dykema, C. Köstler and P. Skoufranis for discussions on topics linked to this work.

\section{Preliminaries}\label{sec:preliminaries}

\subsection{Bi-free probability}\label{subsec:bifreeness}

This subsection is devoted to recalling basic facts concerning opera\-tor-valued bi-free probability as developed in \cite{charlesworth2014combinatorics} assuming that the reader has some basic knowledge in free probability. Let us start by defining the abstract setting. We will denote by $B^{op}$ the opposite algebra of $B$, i.e. with the reversed product.

\begin{de}\label{de:bbspace}
A \emph{$B$-$B$-noncommutative probability space} is a triple $(\A, \E, \varepsilon)$, where $\A$ and $B$ are unital algebras, $\varepsilon : B\otimes B^{op} \rightarrow \A$ is a unital homomorphism whose restrictions to $B\otimes 1$ and $1\otimes B^{op}$ are injective and $\E : A\rightarrow B$ is a linear map satisfying
\begin{eqnarray*}
\EE(\varepsilon(b_{1}\otimes b_{2})T) & = & b_{1}\EE(T)b_{2} \\
\EE(T\varepsilon(b\otimes 1)) & = & \EE(T\varepsilon (1\otimes b))
\end{eqnarray*}
for all $b, b_{1}, b_{2}\in B$ and $T\in \A$. In this context, the \emph{left} and \emph{right} subalgebras of $\A$ are defined as
\begin{eqnarray*}
\A_{\ell} & = & \{T\in \A, T\varepsilon(1\otimes b) = \varepsilon(1\otimes b)T \text{ for all } b\in B^{op}\} \\
\A_{r} & = & \{T\in \A, T\varepsilon(b\otimes 1) = \varepsilon(b\otimes 1)T \text{ for all } b\in B\}.
\end{eqnarray*}
\end{de}

Note that elements in the left algebra commute with the right copy of $B$ (i.e. $\varepsilon(1\otimes B^{op})$) and elements in the right algebra commute with the left copy of $B$ (i.e. $\varepsilon(B\otimes 1)$). The second compatibility condition in Definition \ref{de:bbspace} may seem surprising since it has no counterpart in the definition of a usual operator-valued probability space. It comes from the following concrete example of a $B$-$B$-noncommutative probability space :

\begin{ex}
A \emph{$B$-$B$-bimodule with specified $B$-vector state} is a triple $(X, \mathring{X}, P)$, where $X$ is a direct sum of $B$-$B$-bimodules
\begin{equation*}
X = B\oplus\mathring{X}
\end{equation*}
and $P : X\rightarrow B$ is the linear map $b\oplus x\mapsto b$. There is a morphism from $B\otimes B^{op}$ to the space $\LL(X)$ of all linear maps on $X$ defined by $\varepsilon(b_{1}\otimes b_{2}) = L_{b_{1}}R_{b_{2}} = R_{b_{2}}L_{b_{1}}$, where $L$ and $R$ denote respectively the left and right action of $B$ on $X$. Let us define a map $\E_{B} : \LL(X)\rightarrow B$ by
\begin{equation*}
\E_{B}(T) = P(T(1_{B}\oplus 0)).
\end{equation*}
Then, $(\LL(X), \E_{B}, \varepsilon)$ is a $B$-$B$-noncommutative probability space. Its left and right algebras are denoted respectively by $\LL_{\ell}(X)$ and $\LL_{r}(X)$. Note that by construction, $\E_{B}(T\varepsilon(b\otimes 1)) = \E_{B}(TL_{b}) = \E_{B}(TR_{b}) = \E_{B}(\varepsilon(1\otimes b))$ for all $b\in B$.
\end{ex}

By \cite[Thm 3.2.4]{charlesworth2014combinatorics}, this example is canonical in the sense that any $B$-$B$-noncommutative probability space can be faithfully represented as operators on a $B$-$B$-bimodule with specified $B$-vector state. Moreover, these objects admit a natural free product construction.

\begin{de}
Let $(X_{j}, \mathring{X}_{j}, P_{j})_{j}$ be a family of $B$-$B$-bimodules with specified $B$-vector states. The vector space
\begin{equation*}
\mathring{X} = \sum_{k=1}^{+\infty}\bigoplus_{i_{1}\neq \dots\neq i_{k}}\mathring{X}_{i_{1}}\underset{B}{\otimes} \dots\underset{B}{\otimes} \mathring{X}_{i_{k}}
\end{equation*}
inherits a $B$-$B$-bimodule structure so that, setting $X = B\oplus \mathring{X}$ and $P(b\oplus x) = b$, we have a $B$-$B$-bimodule with specified $B$-vector state $(X, \mathring{X}, P)$, called the \emph{free product} of the family $(X_{j}, \mathring{X}_{j}, P_{j})_{j}$.
If $\mathring{X}(j)$ denotes the direct sum of all the tensor products
\begin{equation*}
\mathring{X}_{i_{1}}\underset{B}{\otimes} \dots\underset{B}{\otimes} \mathring{X}_{i_{k}}
\end{equation*}
with $i_{1}\neq j$, there is a natural isomorphism
\begin{equation*}
V_{j} : X_{j}\underset{B}{\otimes} (B\oplus\mathring{X}(j)) \longrightarrow X.
\end{equation*}
Using it, we can define the \emph{left representation} $\lambda_{j} : \LL(X_{j})\rightarrow \LL_{\ell}(X)$ by
\begin{equation*}
\lambda_{j}(T) = V_{j}(T\otimes \ii)V_{j}^{-1}.
\end{equation*}
One can similarly define the \emph{right representation} $\rho_{j} : \LL(X_{j})\rightarrow \LL_{r}(X)$.
\end{de}

We are now ready for the definition of bi-freeness with amalgamation.

\begin{de}\label{de:bifreeness}
A pair of algebras $(C^{\ell}, C^{r})$ in a $B$-$B$-noncommutative probability space $(\A, \EE, \varepsilon)$ is a \emph{pair of $B$-faces} if
\begin{equation*}
\begin{array}{ccccc}
\varepsilon(B\otimes 1) & \subset & C^{\ell} & \subset & \A_{\ell} \\
\varepsilon(1\otimes B^{op}) & \subset & C^{r} & \subset & \A_{r}. \\
\end{array}
\end{equation*}
A pair of random variables $(x^{\ell}, x^{r})$ is a \emph{$B$-pair} if the algebra generated by $x^{\ell}$ and $\varepsilon(B\otimes 1)$ and the algebra generated by $x^{r}$ and $\varepsilon(1\otimes B^{op})$ form a pair of $B$-faces, or equivalently if $x^{\ell}\in \A_{\ell}$ and $x^{r}\in \A_{r}$.

A family of pairs of $B$-faces $(C_{j}^{\ell}, C_{j}^{r})_{j}$ is said to be \emph{bi-free with amalgamation over $B$} if there exist $B$-$B$-bimodules with specified $B$-vector states $(X_{j}, \mathring{X}_{j}, p_{j})$ for each $j$ together with $B$-morphisms
\begin{eqnarray*}
\ell_{j} : C_{j}^{\ell} & \rightarrow & \LL_{\ell}(X_{j}) \\
r_{j} : C_{j}^{r} & \rightarrow & \LL_{r}(X_{j})
\end{eqnarray*}
such that the joint distribution of  $(C_{j}^{\ell}, C_{j}^{r})_{j}$ with respect to $\E$ is the same as the joint distribution of  $(\lambda_{j}\circ\ell_{j}(C_{j}^{\ell}), \rho_{j}\circ r_{j}(C_{j}^{r}))_{j}$ with respect to the expectation $\E_{B}$ on $\LL(\ast_{j}X_{j})$. A family of $B$-pairs $(x_{j}^{\ell}, x_{j}^{r})_{j}$ is said to be bi-free if the family of pairs of $B$-faces that they generate are bi-free.
\end{de}

The proof of de Finetti theorems in free probability usually involves the combinatorial structure of the joint distributions of the variables. From this point of view, Definition \ref{de:bifreeness} is not the best suited. We therefore now recall from \cite{charlesworth2014combinatorics} the equivalent combinatorial description of bi-freeness with amalgamation. For freeness, the combinatorics are ruled by the so-called \emph{noncrossing partitions} (see for instance \cite{nica2006lectures} for a detailed account). Since they will also play a crucial role in this work, we recall some definitions.

\begin{de}
By a \emph{partition} we mean a partition $\pi$ of the finite set $\{1, \dots, n\}$ for some $n\in \N$ into a family of disjoint subsets whose union is the whole set. Each of these subsets will be called a \emph{block} of $\pi$. If $\pi$ and $\sigma$ are two partitions of the same set, we write $\pi\leqslant \sigma$ if each block of $\pi$ is contained in a block of $\sigma$. A partition $\pi$ is said to be \emph{crossing} if there exist four integers $i_{1} < i_{2} < i_{3} < i_{4}$ such that $i_{1}, i_{3}$ are in the same block, $i_{2}, i_{4}$ are in the same block but the four integers are not in the same block. Otherwise, the partition is said to be \emph{noncrossing}.
\end{de}

We may represent a partition $\pi$ by drawing $n$ points on a row, labelled by integers from left to right, and connecting by a line the points whose labels are in the same block of $\pi$. For instance the partition $\pi = \{\{1, 2\}, \{3, 5, 9\}, \{4, 7\}, \{6, 8\}\}$ of $\{1, \dots, 9\}$ will be drawn as
\begin{center}
\begin{picture}(13,4)
\put(0,1){$\pi = $}
\put(3,0){\usebox{\partpinumber}}
\end{picture}
\end{center}

Given any monomial in operators belonging to pairs of faces, we can consider the associated "left-right colouring", i.e. a sequence of $\ell$ and $r$'s. This sequence $\chi$ gives rise to a permutation $s_{\chi}$ which will be the crucial combinatorial tool. In fact, the correct family of partitions to consider is noncrossing partitions "twisted" by $s_{\chi}$. The precise procedure was discovered by M. Mastnak and A. Nica in \cite{mastnak2014double} :

\begin{de}\label{de:binoncrossing}
Let $n\in \N$, let $\chi\in \{\ell, r\}^{n}$ (which will be seen as a function $\{1, \dots, n\}\rightarrow \{\ell, r\}$) and set
\begin{eqnarray*}
\chi^{-1}(\ell) & = & \{i_{1} < \dots < i_{k}\} \\
\chi^{-1}(r) & = & \{i_{k+1} > \dots > i_{n}\}.
\end{eqnarray*}
We define a permutation $s_{\chi}\in S_{n}$ by $s_{\chi}(t) = i_{t}$. A partition $\pi$ is said to be \emph{bi-noncrossing relative to $\chi$} if $s_{\chi}^{-1}(\pi)$ (the partition obtained by applying $s_{\chi}^{-1}$ to the blocks of $\pi$, see \cite[Def 2.1.1]{charlesworth2014combinatorics}) is a noncrossing partition. The set of bi-noncrossing partitions relative to $\chi$ will be denoted by $BNC(\chi)$. The partition in $BNC(\chi)$ with only one block will be denoted by $1_{\chi}$.
\end{de}

We give a pictorial example for this, which is taken from \cite[Ex 5.1.2]{charlesworth2014combinatorics}.

\begin{ex}\label{ex:binc}
Let $\chi=\{\ell, r, \ell, \ell, r, r, \ell, r, r\}$ and let $\pi = \{\{1, 2\}, \{3, 5, 9\}, \{4, 7\}, \{6, 8\}\}\in BNC(\chi)$. The permutation $s_{\chi}$ is given by
\begin{equation*}
\left(
\begin{array}{ccccccccc}
1 & 2 & 3 & 4 & 5 & 6 & 7 & 8 & 9 \\
1 & 3 & 4 & 7 & 9 & 8 & 6 & 5 & 2
\end{array}
\right).
\end{equation*}
Then, $s_{\chi}^{-1}(\pi) = p = \{\{1, 9\}, \{2, 5, 8\}, \{3, 4\}\}, \{6, 7\}\}\in NC(9)$ and $\pi$ is bi-noncrossing relative to $\chi$. This equality can be seen on partitions in the following way : on top of $\pi$, we draw lines pulling all the left points to the left and all the right points to the right. Then, we keep the left lines straight while we permute all the right ones.
\begin{center}
\begin{picture}(13,10)
\put(0.5,1){$\pi = $}
\put(2,0){\usebox{\partpi}}
\put(0.5,6){$s_{\chi} = $}
\put(2,4){\usebox{\partpullleftright}}
\put(2,6.3){\usebox{\partreverseorder}}
\end{picture}
\end{center}
Composing the two upper partitions in this picture gives a pictorial representation of the permutation $s_{\chi}$. Since it is placed above $\pi$ (it is "read from bottom to top"), the resulting partition is $s_{\chi}^{-1}(\pi)$ :
\begin{center}
\begin{picture}(13,4)
\put(0,1){$s_{\chi}^{-1}(\pi) = $}
\put(3,0){\usebox{\partpermutedpi}}
\end{picture}
\end{center}
\end{ex}

I. Charlesworth, B. Nelson and P. Skoufranis introduced another pictorial representation for bi-noncrossing partitions, drawing the points on vertical lines instead of horizontal ones. Our point in keeping the horizontal drawing is to make the connection with the combinatorics of quantum groups easier, as will appear later on. Mimicking free probability, one would like to define "bi-noncrossing cumulants" $\kappa^{\chi}_{1_{\chi}}$ by induction using a defining formula of the type
\begin{equation}\label{eq:momentcumulant}
\E(T_{1} \dots T_{n}) = \sum_{\sigma\in BNC(\chi)}\kappa_{\sigma}^{\chi}(T_{1}, \dots, T_{n}),
\end{equation}
where $\kappa^{\chi}_{\sigma}$ should be computable using only elements of the form $\kappa^{\chi'}_{1_{\chi'}}$ for restrictions $\chi'$ of $\chi$. To do this, the authors of \cite{charlesworth2014combinatorics} first define "bi-moment functions" $\mathcal{E}_{\pi}$ and then use Möbius inversion. As we will see, these moment functions look like the usual ones except that they take care of the left-right structure. To avoid confusion, they will be denoted by a curly letter $\mathcal{E}$ while we will denote by $\E_{p}$, for a noncrossing partition $p$, the usual nested moment function as defined e.g. in \cite[Def 2.1.1]{speicher1998combinatorial}. For the sake of simplicity, we will from now on write $L_{b} = \varepsilon(b\otimes 1)$ and $R_{b} = \varepsilon(1\otimes b)$.

\begin{de}\label{de:moments}
Let $\chi\in \{\ell, r\}^{n}$, let $\pi\in BNC(\chi)$ and choose an element $T_{j_{i}}^{\chi(i)}\in C_{j_{i}}^{\chi(i)}$ for every $1\leqslant i\leqslant n$. We define an element
\begin{equation*}
\mathcal{E}_{\pi}(T_{j_{1}}^{\chi(1)}, \dots, T_{j_{n}}^{\chi(n)})\in B
\end{equation*}
in three steps :
\begin{enumerate}
\item Permute the tuple $(T_{j_{1}}^{\chi(1)}, \dots, T_{j_{n}}^{\chi(n)})$ to $(T_{j_{s_{\chi}(1)}}^{\chi\circ s_{\chi}(1)}, \dots, T_{j_{s_{\chi}(n)}}^{\chi\circ s_{\chi}(n)})$.
\item Apply the expectation $\E_{s_{\chi}^{-1}(\pi)}$ to this tuple. Using the properties of usual multiplicative functions, this can be written as a nesting of the expectation $\E$, since $s_{\chi}^{-1}(\pi)$ is noncrossing.
\item Blockwise, permute back the elements according to $s_{\chi'}^{-1}$, where $\chi'$ is the restriction of $\chi$ to the block. Then replace each nested expectation $b$ by $L_{b}$ if the leftmost element of the block is a left element and by $R_{b}$ if the first element of the block is a right element.
\end{enumerate}
\end{de}

Note that the functions $\mathcal{E}_{\pi}$ are completely determined by $\E$. Let us illustrate this seemingly complicated definition on the same partition as in Example \ref{ex:binc}.

\begin{ex}
We consider $\chi$ and $\pi$ as in Example \ref{ex:binc}. Let us compute
\begin{equation*}
\mathcal{E}_{\pi}(T_{1}, T_{2}, T_{3}, T_{4}, T_{5}, T_{6}, T_{7}, T_{8}, T_{9}).
\end{equation*}
\begin{enumerate}
\item Permute the elements to $(T_{1}, T_{3}, T_{4}, T_{7}, T_{9}, T_{8}, T_{6}, T_{5}, T_{2})$.
\item Compute
\begin{equation*}
\E_{p}(T_{1}, T_{3}, T_{4}, T_{7}, T_{9}, T_{8}, T_{6}, T_{5}, T_{2}) = \E\left(T_{1}\left(\E(T_{3}\left(\E(T_{4}T_{7})T_{9}\E(T_{8}T_{6})T_{5}\right) \right)T_{2}\right).
\end{equation*}
\item Permute back the elements blockwise to
\begin{equation*}
\E\left(T_{1}\left(\E(T_{3}\left(\E(T_{4}T_{7})T_{5}\E(T_{6}T_{8})T_{9}\right) \right)T_{2}\right)
\end{equation*}
and then replace by the corresponding action of $B$ to get
\begin{equation*}
\mathcal{E}_{\pi}(T_{1}, T_{2}, T_{3}, T_{4}, T_{5}, T_{6}, T_{7}, T_{8}, T_{9}) = \E\left(T_{1}L_{\left(\E(T_{3}\left(L_{\E(T_{4}T_{7})}T_{5}R_{\E(T_{6}T_{8})}T_{9}\right)\right)}T_{2}\right).
\end{equation*}
\end{enumerate}
Because $T_{2}$ is a right element, it commutes with every $L_{b}$, so that we recover \cite[Ex 5.1.2]{charlesworth2014combinatorics}.
\end{ex}

\begin{rem}
The definition of the \emph{operator-valued bi-moment function} in \cite{charlesworth2014combinatorics} is different from Definition \ref{de:moments} and much more involved. However, \cite[Thm 5.1.4]{charlesworth2014combinatorics} asserts that these bi-free moment functions are completely determined by $\E$ and the properties of bi-multiplicative functions. This means that they coincide with Definition \ref{de:moments}. We took here advantage of this deep result to give the "easy" definition of the bi-moment functions.
\end{rem}

Note that for any $\chi\in \{\ell, r\}^{n}$, $s_{\chi}^{-1}$ is an order preserving bijection between $BNC(\chi)$ and $NC(n)$ so that in particular the lattice structure is preserved and the Möbius function on $BNC(\chi)$ is simply given by
\begin{equation*}
\mu_{BNC(\chi)}(\pi, \sigma) = \mu_{NC}(s_{\chi}^{-1}(\pi), s_{\chi}^{-1}(\sigma)).
\end{equation*}
Thus, we can define \emph{bi-noncrossing cumulants} (or $(\ell, r)$-cumulants) by the formula
\begin{equation*}
\kappa^{\chi}_{\pi}(T_{1}, \dots, T_{n}) = \sum_{\sigma\in BNC(\chi)}\mu_{BNC(\chi)}(\sigma, \pi)\mathcal{E}_{\sigma}(T_{1}, \dots, T_{n}),
\end{equation*}
where $\chi$ is the colouring of the tuple $(T_{1}, \dots, T_{n})$. This formula can be inverted to yield the moment-cumulant formula of Equation \eqref{eq:momentcumulant}. That these cumulants are the right combinatorial notion for bi-freeness with amalgamation is the content of \cite[Thm 8.1.1]{charlesworth2014combinatorics}. This theorem states that a family of pairs of $B$-faces is bi-free with amalgamation if and only if the cumulants $\kappa_{1_{\chi}}^{\chi}$ vanishes as soon as its arguments do not all belong to the same pair. Combining this with the general properties of bi-multiplicative functions gives a stronger statement which is the one we need. Let $J = (j_{1}, \dots, j_{n})$ be a tuple of integers. The \emph{kernel of $J$} is the partition $\ker(J)$ of $\{1, \dots, n\}$ where two integers $k$ and $k'$ belong to the same block if and only if $j_{k} = j_{k'}$.

\begin{thm}[Charlesworth -- Nelson -- Skoufranis]\label{thm:vanishingcumulants}
Let $(\A, \E, \varepsilon)$ be a $B$-$B$-noncommutative probability space and let $(C_{j}^{\ell}, C_{j}^{r})_{j}$ be a family of pairs of $B$-faces. Then, it is bi-free with amalgamation over $B$ if and only if for all $n\in \N$, any $\chi\in \{\ell, r\}^{n}$, any $\pi\in BNC(\chi)$, any tuple of integers $j_{1}, \dots, j_{n}$ and any choice of $T_{j_{i}}^{\chi(i)}\in C_{j_{i}}^{\chi(i)}$,
\begin{equation*}
\kappa^{\chi}_{\pi}\left(T_{j_{1}}^{\chi(1)} , \dots, T_{j_{n}}^{\chi(n)}\right) = 0
\end{equation*}
as soon as $\pi\nleqslant \ker(J)$.
\end{thm}

\subsection{The quantum permutation group}

It has been known since \cite{kostler2009noncommutative} that the symmetries characterizing freeness with amalgamation are given by the \emph{quantum permutation group} $S_{N}^{+}$. This is a \emph{compact quantum group} in the sense of S.L. Woronowicz \cite{woronowicz1995compact} and was introduced by S. Wang in \cite{wang1998quantum}. Since $S_{N}^{+}$ will also play the role of quantum symmetries in this work, we recall hereafter some basic facts about it. Let $C(S_{N}^{+})$ be the universal unital C*-algebra generated by $N^{2}$ self-adjoint projections $u_{ij}$ satisfying, for all $1\leqslant k\leqslant N$,
\begin{equation}\label{eq:defsn}
\sum_{i=1}^{N}u_{ik} = 1 = \sum_{j=1}^{N}u_{kj}.
\end{equation}
Viewing $u = (u_{ij})_{1\leqslant i, j\leqslant N}$ as a matrix, this means that the sum of the coefficients on any row or column is $1$. This implies in particular that any two elements on the same row or the same column are orthogonal. Moreover, Equation \eqref{eq:defsn} implies that the matrix $u$ is \emph{orthogonal} in the sense that its coefficients are self-adjoint and
\begin{equation*}
{}^{t}uu = \Id_{M_{N}(C(S_{N}^{+}))} = u{}^{t}u.
\end{equation*}
The C*-algebra $C(S_{N}^{+})$ can be endowed with a \emph{compact quantum group} structure (see \cite{woronowicz1995compact} for details) thanks to the coproduct
\begin{equation*}
\D(u_{ij}) = \sum_{k=1}^{N}u_{ik}\otimes u_{kj},
\end{equation*}
where $\otimes$ denotes the spatial tensor product of C*-algebras. The study of the representation theory of $S_{N}^{+}$ by T. Banica in \cite{banica1999symmetries} enables to compute some polynomials in the coefficients of $u$. We will later make use of some of these computations, which we gather in the next proposition. If $p$ is any partition, we set $\delta_{p}(J) = 1$ if $p\leqslant \ker(J)$ and $\delta_{p}(J) = 0$ otherwise.

\begin{prop}\label{prop:vanishingsn}
Let $n\in \N$ and let $p$ be a noncrossing partition. Then, for any $J = (j_{1}, \dots, j_{n})$,
\begin{equation*}
\sum_{\underset{p\leqslant \ker(I)}{I = (i_{1}, \dots, i_{n})}}u_{i_{1}j_{1}}\dots u_{i_{n}j_{n}} = \delta_{p}(J).1_{C(S_{N}^{+})}.
\end{equation*}
\end{prop}

Let us consider the classical permutation group $S_{N}$ represented on $\C^{N}$ as permutation matrices. For $1\leqslant i, j\leqslant N$, let $v_{ij} : S_{N} \rightarrow \C$ be the function sending a permutation matrix $\sigma$ to its $(i, j)$-th coefficient. Then, the functions $v_{ij}$ generate the algebra $C(S_{N})$ of all (continuous) functions on $S_{N}$. Moreover, the family $v_{ij}$ obviously satisfies the defining relations of $C(S_{N}^{+})$. Hence, by the definition of a universal C*-algebra, there is a surjective $*$-homomorphism $\Phi : C(S_{N}^{+})\longrightarrow C(S_{N})$ completely determined by $\Phi(u_{ij}) = v_{ij}$. Moreover, one can prove that $\Phi$ respects the quantum group structure of $S_{N}^{+}$ and the group structure of $S_{N}$ and that $S_{N}$ is the biggest classical group admitting such a quotient map. This is one of the reasons why $S_{N}^{+}$ can be seen as a quantum version of the permutation group.

\section{Quantum invariance and bi-freeness for families of pairs}\label{sec:DeFinetti}

In this section, we will study the consequences of quantum bi-invariance in two different settings, first assuming some bi-noncommutative structure on the pairs and then taking a more general approach. This first requires the introduction of a specific family of actions of $S_{N}^{+}$  on families of pairs.

\subsection{Linear action of quantum permutation groups}\label{subsec:linearaction}

It is not clear at first how $S_{N}^{+}$ can be used to define the symmetries of bi-freeness. The path we take here is to consider, instead of a usual quantum group action, a \emph{linear action} of $S_{N}^{+}$. A similar idea was used by W. Liu in \cite{liu2014noncommutative} to characterize boolean independence, even though he was using quantum semigroups instead of quantum groups.

\begin{de}
A linear action of $S_{N}^{+}$ on a vector space $V$ is a linear map $\beta : V\rightarrow V\otimes C(S_{N}^{+})$ such that
\begin{equation*}
(\beta\otimes \ii)\circ\beta = (\ii\otimes \D)\circ\beta.
\end{equation*}
If $\varphi$ is a linear functional on $V$, then it is said to be \emph{invariant} under $\beta$ if $(\varphi\otimes \ii)\circ\beta = \varphi.1_{C(S_{N}^{+})}$.
\end{de}

Let us insist that $\beta$ is not assumed to be multiplicative, hence is not a quantum group action in the usual sense. We will have to deal in the sequel with moments of monomials, i.e. of products of elements indexed by tuples. To make things easier we introduce some shorthand notations. If $(x_{j}^{\ell}, x_{j}^{r})_{j}$ is a family of pairs of random variables then, for any $I = (i_{1}, \dots, i_{n})$ and any $\chi \in \{\ell, r\}^{n}$, we set
\begin{equation*}
x_{I}^{\chi} = x_{i_{1}}^{\chi(1)}\dots x_{i_{n}}^{\chi(n)}.
\end{equation*}
If $(u_{ij})_{ij}$ is a matrix whose coefficients are operators then, for any $I = (i_{1}, \dots, i_{n})$, any $J = (j_{1}, \dots, j_{n})$ and any $\chi \in \{\ell, r\}^{n}$, we set 
\begin{equation*}
u_{IJ} = u_{i_{1}j_{1}}\dots u_{i_{n}j_{n}} \text{ and } u_{IJ}^{\chi} = u_{s_{\chi}(I)s_{\chi}(J)},
\end{equation*}
where $s_{\chi}(I) = (i_{s_{\chi}(1)}, \dots, i_{s_{\chi}(n)})$ and similarly for $s_{\chi}(J)$. Eventually, when summing on tuples, we will set
\begin{equation*}
\sum_{I}^{N} = \sum_{i_{1}, \dots, i_{n} = 1}^{N}.
\end{equation*}

\begin{rem}
With the notations above, the formula for the coproduct of a monomial of coefficients of $u$ is given by
\begin{equation*}
\D(u_{IJ}) = \sum_{K}^{N}u_{IK}\otimes u_{KJ}.
\end{equation*}
\end{rem}

Our quantum symmetries for bi-freeness will be given by the following linear action of $S_{N}^{+}$ :

\begin{de}\label{de:permutationaction}
Let $(\A, \varphi)$ be a noncommutative probability space, let $(x_{j}^{\ell}, x_{j}^{r})_{1\leqslant j\leqslant N}$ be a finite family of pairs of random variables in $\A$ and let $\M\subset \A$ be the algebra that they generate. The \emph{bi-free quantum permutation action}, is the linear action $\beta_{N}$ of $S_{N}^{+}$ on $\M$ given by
\begin{equation}\label{eq:linearaction}
\beta_{N}(x_{J}^{\chi}) = \sum_{I}^{N} x_{I}^{\chi}\otimes u_{IJ}^{\chi} .
\end{equation}
\end{de}

Let $\alpha_{N}$ be the multiplicative action of $S_{N}^{+}$ on $\M$ by "quantum permutation of pairs", i.e. the unique algebra homomorphism $\M \rightarrow \M\otimes C(S_{N}^{+})$ such that
\begin{equation*}
\alpha_{N}(x_{j}^{\ell}) = \sum_{i=1}^{N}x_{i}^{\ell}\otimes u_{ij} \text{ and }
\alpha_{N}(x_{j}^{r}) = \sum_{i=1}^{N}x_{i}^{r}\otimes u_{ij}.
\end{equation*}
We have the following description of $\beta_{N}$ using $\alpha_{N}$ for a monomial $x_{I}^{\chi}$ :  first permute the variables of the monomial using $s_{\chi}$, then apply $\alpha_{N}$ and eventually permute back the variables of the monomials using $s_{\chi}^{-1}$. Comparing this with the process described in Definition \ref{de:moments} suggests that $\beta_{N}$ is a natural candidate for the quantum symmetries of bi-freeness. Note also that this permutation in the definition of $\beta_{N}$ is precisely what prevents it from being multiplicative.

\begin{lem}\label{prop:nondegenerate}
The map $\beta_{N}$ is a linear action of $S_{N}^{+}$ on the vector space $\M$.
\end{lem}

\begin{proof}
Let $J$ be a $n$-tuple of indices and let $\chi\in \{\ell, r\}^{n}$. Then,
\begin{eqnarray*}
(\beta_{N}\otimes \ii)\circ\beta_{N}(x_{J}^{\chi}) & = & \sum_{I}^{N}\sum_{K}^{N}x_{K}^{\chi}\otimes u_{KI}^{\chi}\otimes u_{IJ}^{\chi} \\
& = & \sum_{K}^{N}x_{K}^{\chi}\otimes\left(\sum_{I}^{N} u_{s_{\chi}(K)s_{\chi}(I)}\otimes u_{s_{\chi}(I)s_{\chi}(J)}\right) \\
& = & \sum_{K}^{N}x_{K}^{\chi}\otimes\left(\sum_{I'=1}^{N} u_{s_{\chi}(K)I'}\otimes u_{I's_{\chi}(J)}\right) \\
& = & (\ii\otimes \D)\circ\beta_{N}(x_{J}^{\chi})
\end{eqnarray*}
proving that $\beta_{N}$ is a linear action.
\end{proof}

We will be interested in invariance under the linear action $\beta_{N}$, so let us give a name to this phenomenon.

\begin{de}
Let $(\A, \varphi)$ be a noncommutative probability space. A finite sequence of pairs $(x_{j}^{\ell}, x_{j}^{r})_{1\leqslant j\leqslant N}$ in $\A$ is said to be \emph{quantum bi-exchangeable} if $\varphi$ is invariant under $\beta_{N}$, i.e.
\begin{equation*}
\varphi(x_{J}^{\chi}).1_{C(S_{N}^{+})} = \sum_{I}^{N}\varphi(x_{I}^{\chi})u_{IJ}^{\chi}.
\end{equation*}
\end{de}

\begin{rem}
Usually, such actions are extended to the von Neumann algebra generated by $\M$ in the GNS representation of $\varphi$. However, we cannot apply \cite[Thm 3.3]{curran2009quantum} to do this here, because the action $\beta_{N}$ is not multiplicative. This is the reason why we will have to deal only with the algebra generated by the variables all along.
\end{rem}

Our first positive result is that quantum bi-invariance is compatible with bi-freeness with amalgamation. This is the "easy" part of the de Finetti theorem and justifies the introduction of our notion of invariance.

\begin{prop}\label{prop:bifinettieasy}
Let $(\A, \E, \varepsilon)$ be a $B$-$B$-noncommutative probability space and let $(x_{j}^{\ell}, x_{j}^{r})_{1\leqslant j\leqslant N}$ be a family of $B$-pairs which are bi-free and identically distributed with amalgamation over $B$. Let $\varphi$ be a state on $\A$ such that, for all $x\in \A$,
\begin{equation}\label{eq:compatibility}
\varphi\circ\varepsilon(\E(x)\otimes 1) = \varphi(x) = \varphi\circ\varepsilon(1\otimes \E(x)).
\end{equation}
Then, $(x_{j}^{\ell}, x_{j}^{r})_{1\leqslant j\leqslant N}$ is a quantum bi-exchangeable sequence of pairs in $(\A, \varphi)$.
\end{prop}

\begin{proof}
We first prove an invariance property for $\E$. Let $J$ be a $n$-tuple and let $\chi\in\{\ell, r\}^{n}$. Using Equation \eqref{eq:momentcumulant} we have
\begin{eqnarray*}
(\E\otimes \ii)\circ\beta_{N}(x_{J}^{\chi}) & = & \sum_{I}^{N}\E(x_{I}^{\chi})\otimes u_{IJ}^{\chi} \\
& = & \sum_{I}^{N}\sum_{\pi\in BNC(\chi)}\kappa^{\chi}_{\pi}(x_{I}^{\chi})\otimes u_{IJ}^{\chi}.
\end{eqnarray*}
By Theorem \ref{thm:vanishingcumulants}, $\kappa^{\chi}_{\pi}(x_{I}^{\chi}) = 0$ unless $\pi\leqslant \ker(I)$. Thus, using the fact that the elements are identically distributed, the sum can be rewritten as
\begin{equation*}
\sum_{\pi\in BNC(\chi)}\kappa^{\chi}_{\pi}\otimes \left(\sum_{\underset{\pi\leqslant \ker(I)}{I}}^{N}u_{IJ}^{\chi}\right),
\end{equation*}
where $\kappa^{\chi}_{\pi}$ is the common value of $\kappa^{\chi}_{\pi}(x_{I}^{\chi})$ for all $I$'s satisfying $\pi\leqslant \ker(I)$. By Proposition \ref{prop:vanishingsn},
\begin{equation*}
\sum_{\underset{\pi\leqslant \ker(I)}{I}}^{N}u_{IJ}^{\chi} = \sum_{\underset{s_{\chi}^{-1}(\pi)\leqslant \ker(I')}{I'}}^{N}u_{I's_{\chi}(J)} = \delta_{s_{\chi}^{-1}(\pi)}(s_{\chi}(J)).1_{C(S_{N}^{+})}
\end{equation*}
because $s_{\chi}^{-1}(\pi)$ is noncrossing. Thus,
\begin{equation*}
(\E\otimes \ii)\circ\beta_{N}(x_{J}^{\chi}) = \sum_{\pi\in BNC(\chi)}\kappa^{\chi}_{\pi}\otimes \delta_{s_{\chi}^{-1}(\pi)}(s_{\chi}(J)).1_{C(S_{N}^{+})}.
\end{equation*}
Let us now compute $\E(x_{J}^{\chi})\otimes 1_{C(S_{N}^{+})}$ and compare it with this. Because the pairs are bi-free and identically distributed with amalgamation over $B$, Theorem \ref{thm:vanishingcumulants} implies that
\begin{equation*}
\kappa^{\chi}_{\pi}(x_{J}^{\chi}) = \delta_{\pi}(J)\kappa^{\chi}_{\pi},
\end{equation*}
yielding
\begin{equation*}
\E(x_{J}^{\chi})\otimes 1_{C(S_{N}^{+})} = \sum_{\pi\in BNC(\chi)}\kappa_{\pi}^{\chi}(x_{J}^{\chi})\otimes 1_{C(S_{N}^{+})} = \sum_{\pi\in BNC(\chi)}\kappa^{\chi}_{\pi}\otimes \delta_{\pi}(J).1_{C(S_{N}^{+})}.
\end{equation*}
Noticing that $\delta_{\pi}(J) = \delta_{s_{\chi}^{-1}(\pi)}(s_{\chi}(J))$ shows the invariance of $\E$. Now, the compatibility condition of Equation \eqref{eq:compatibility} yields
\begin{eqnarray*}
(\varphi\otimes \ii)\circ\beta_{N}(x_{I}^{\chi}) & = & (\varphi\circ\varepsilon\circ(1\otimes \E)\otimes \ii)\circ\beta_{N}(x_{I}^{\chi}) \\
& = & ((\varphi\circ\varepsilon)\otimes \ii)(1\otimes (\E\otimes \ii)\circ\beta_{N}(x_{I}^{\chi})) \\
& = & \varphi\circ\varepsilon(1\otimes \E(x_{I}^{\chi}))\otimes 1_{C(S_{N}^{+})} \\
& = & \varphi(x_{I}^{\chi}).1_{C(S_{N}^{+})}.
\end{eqnarray*}
\end{proof}

Proposition \ref{prop:bifinettieasy} suggests that quantum bi-invariance as we defined it can play the role of quantum symmetries for bi-free probability. The next task is therefore to prove a converse statement showing that quantum bi-invariant families of pairs are bi-free with amalgamation. As we will see, several issues arise when studying this problem and we can only partially solve it.

\subsection{An alternate characterization of bi-freeness}\label{subsec:weakdefinetti}

In this subsection, we will give a weak converse to Proposition \ref{prop:bifinettieasy}. Before stating it, let us introduce some notations. From now on, we fix a $B$-$B$-noncommutative probability space $(\A, \E, \varepsilon)$ and a family of $B$-pairs $(x_{j}^{\ell}, x_{j}^{r})_{j}$. We will denote by $C_{j}^{\ell}$ (resp. $C_{j}^{r}$) the subalgebra of $\A$ generated by $(x_{j}^{\ell})_{j}$ and $\varepsilon(B\otimes 1)$ (resp. by $(x_{j}^{\ell})_{j}$ and $\varepsilon(1\otimes B^{op})$). Our statement will be a weak converse because we will make strong assumptions on the variables. To simplify the statement, let us give names to these assumptions.

\begin{de}
The family of $B$-pairs $(x_{j}^{\ell}, x_{j}^{r})_{1\leqslant j\leqslant N}$ is said to be \emph{strongly quantum bi-invariant} if for any $j_{1}, \dots, j_{n}$, any $\chi\in \{\ell, r\}^{n}$ and any $b_{1}, \dots, b_{n+1}\in B$,
\begin{equation*}
\E\left(T_{b_{1}}x_{j_{1}}^{\chi(1)}T_{b_{2}}\dots x_{j_{n}}^{\chi(n)}T_{b_{n+1}}\right)\otimes 1_{C(S_{N}^{+})} = \sum_{I=1}^{N} \E\left(T_{b_{1}}x_{i_{1}}^{\chi(1)}T_{b_{2}}\dots x_{i_{n}}^{\chi(n)}T_{b_{n+1}}\right)\otimes u_{IJ}^{\chi}
\end{equation*}
where $T$ can be either $L$ or $R$. Moreover, it is said to satisfy the \emph{splitting property} if for any $n$ different indices $j_{1}, \dots, j_{n}$ and elements $X_{j_{i}}^{\chi(i)}\in C_{j_{i}}^{\chi(i)}$,
\begin{equation*}
\E\left(X_{j_{1}}^{\chi(1)}\dots X_{j_{n}}^{\chi(n)}\right) = \prod_{i=1}^{n}\E\left(X_{j_{s_{\chi}(i)}}^{\chi\circ s_{\chi}(i)}\right).
\end{equation*}
\end{de}

We can now characterize bi-freeness in terms of quantum invariance.

\begin{thm}\label{thm:weakdefinetti}
Let $(\A, \E, \varepsilon)$ be a $B$-$B$-noncommutative probability space. A family $(x_{j}^{\ell}, x_{j}^{r})_{1\leqslant j\leqslant N}$ of $B$-pairs is bi-free and identically distributed with amalgamation over $B$ if and only if it is strongly quantum bi-invariant and satisfies the splitting property.
\end{thm}

Note that we are not assuming that the family is infinite here. This proposition is proved by comparing $\E$ with a twisted free product expectation which characterizes bi-freeness with amalgamation. For a tuple $(X_{j_{1}}^{\chi(1)}, \dots, X_{j_{n}}^{\chi(n)})$ we will, with a slight abuse of notations, write $X_{J}^{\chi}$ for both the product of the elements or the tuple itself when it is the argument of a moment or cumulant function. The first step of the proof is to define the twisted expectations $G$.

\begin{de}
Given a tuple $(X_{j_{1}}^{\chi(1)}, \dots, X_{j_{n}}^{\chi(n)})$, we define $G(X_{J}^{\chi})$ as follows :
\begin{enumerate}
\item Replace each occurrence of $L_{b}T$, $TL_{b}$, $R_{b}T$ and $TR_{b}$ by $bT$, $Tb$, $Tb$ and $bT$ respectively (see \cite[Rem 4.2.4]{charlesworth2014combinatorics}).
\item Permute the tuple using $s_{\chi}$.
\item Apply $\E$ to the product of the elements and compute it as if the pairs were free.
\item Eventually, reverse the first two operations.
\end{enumerate}
\end{de} 

For our purpose, the interest of the map $G$ lies in its link to bi-freeness. This is a direct consequence of the results of \cite{charlesworth2014combinatorics}, but since it is not explicitly stated in the way we need there, we give it as a lemma.

\begin{lem}\label{lem:characterization}
The family of $B$-pairs $(x_{j}^{\ell}, x_{j}^{r})_{j}$ is bi-free with amalgamation over $B$ if and only if for all $n\in \N$, all $\chi\in \{\ell, r\}^{n}$ and all $X_{j_{1}}^{\chi(1)}, \dots, X_{j_{n}}^{\chi(n)}$ such that $X_{j_{i}}^{\chi(i)}\in C_{j_{i}}^{\chi(i)}$, we have
\begin{equation*}
\E(X_{J}^{\chi}) = G(X_{J}^{\chi}).
\end{equation*}
\end{lem}

The problem is now to prove that the two linear functionals $\E$ and $G$ are equal. This can be done by showing that they share enough properties to satisfy a uniqueness result. The following lemma gives a sufficient set of properties. Let us denote by $\NN$ the algebra generated by $C_{j}^{\ell}$ and $C_{j}^{r}$ for all $j$.

\begin{lem}\label{lem:uniqueness}
Let $\Phi : \NN \rightarrow B$ be a linear map satisfying, for any $\chi\in \{\ell, r\}^{n}$ and any tuple $J = (j_{1}, \dots, j_{n})$,
\begin{enumerate}
\item $\Phi(X_{J}^{\chi}) = 0$ as soon as $\ker(J) = 0_{\chi}$, with $0_{\chi}$ the partition in $BNC(\chi)$ with all blocks of size one.
\item $\Phi$ is quantum bi-invariant : $\Phi(X_{J}^{\chi})\otimes 1 = \displaystyle\sum_{I}^{N} \Phi(X_{I}^{\chi})\otimes u^{\chi}_{IJ}$.
\end{enumerate}
Then, $\Phi = 0$.
\end{lem}

\begin{proof}
We will follow the scheme of the proof of \cite[Thm 1.1]{kostler2009noncommutative}. Let us fix an integer $n$ and prove that $\Phi(X_{J}^{\chi}) = 0$ for all $J$ of length at most $n$ by decreasing induction, using the induction hypothesis
\begin{center}
$H(r)$ : if $\ker(J)$ has at least $r$ blocks, then $\Phi(X_{J}^{\chi}) = 0$.
\end{center}
If $r=n$, this is property $(1)$. Let us now assume $H(r+1)$ for some $1\leqslant r\leqslant n-1$ and prove $H(r)$. By condition $(2)$, we have
\begin{eqnarray*}
\Phi(X_{J}^{\chi})\otimes 1 & = & \sum_{I}^{N} \Phi(X_{I}^{\chi})\otimes u^{\chi}_{IJ} \\
& = & \sum_{\pi\in P(n)}\sum_{\underset{\ker(I) = \pi}{I=1}}^{N} \Phi(X_{I}^{\chi})\otimes u^{\chi}_{IJ}.
\end{eqnarray*}
By $H(r+1)$, we can restrict in the above sum to partitions $\pi$ having at most $r$ blocks. Moreover, this relation must hold for any array of projections $(u_{ij})_{i, j}$ satisfying the defining relations of $C(S_{N}^{+})$. Using the particular matrix of \cite[Eq 9]{kostler2009noncommutative} we see that we may assume $\ker(J)\leqslant \ker(I)$. Because $\ker(I) = \pi$ and $\pi$ has at most $r$ blocks, this yields $\pi = \ker(J)$, reducing the invariance equation to
\begin{equation*}
\Phi(X_{J}^{\chi})\otimes 1 = \sum_{\underset{\ker(I) = \ker(J)}{I=1}}^{N}\Phi(X_{I}^{\chi})\otimes u_{IJ}^{\chi}.
\end{equation*}
The invariance condition $(2)$ implies that $\Phi(X_{I}^{\chi})$ is invariant under permutation of the pairs in the following sense : for any permutation $\sigma$,
\begin{equation*}
\Phi(X_{\sigma(i_{1})}^{\chi(1)}\dots X_{\sigma(i_{n})}^{\chi(n)}) = \Phi(X_{I}^{\chi}).
\end{equation*}
Thus, $\Phi(X_{I}^{\chi})$ depends only on $\ker(I)$ and we can take it out of the sum to get
\begin{equation*}
\Phi(X_{J}^{\chi})\otimes 1 = \Phi(X_{J}^{\chi})\otimes\left(\sum_{\underset{\ker(I) = \ker(J)}{I=1}}^{N} u_{IJ}^{\chi}\right).
\end{equation*}
Making the change of variables $I' = s_{\chi}(I)$ and setting $J'=s_{\chi}(J)$, the term in parenthesis becomes
\begin{equation*}
\sum_{\underset{\ker(I') = \ker(J')}{I'=1}}^{N}u_{I'J'}.
\end{equation*}
It was proved in \cite[Thm 1.1]{kostler2009noncommutative} that this is not equal to $1_{C(S_{N}^{+})}$, hence $\Phi(X_{J}^{\chi})$ must vanish and $H(r)$ is proved.
\end{proof}

Let us set $\Phi = \E - G$. It is clear that $\Phi$ is a linear map from $\NN$ to $B$ and we have to check the two conditions of Lemma \ref{lem:uniqueness}.

\begin{proof}[Proof of Theorem \ref{thm:weakdefinetti}]
If the pairs are bi-free and identically distributed with amalgamation over $B$, then they are strongly quantum bi-invariant by Proposition \ref{prop:bifinettieasy}. Moreover, they have the splitting property by a direct calculation using the definition of bi-freeness. We therefore only have to prove the converse part of the statement. It is clear that $G$ satisfies the invariance condition $(2)$ of Lemma \ref{lem:uniqueness} by definition. Thus, $\Phi = \E - G$ satisfies condition $(2)$ and we only have to prove that the vanishing condition $(1)$ is satisfied. Assume that $\ker(J) = 0_{\chi}$. By definition of $G$ and of the free expectation,
\begin{equation*}
G\left(X_{J}^{\chi}\right) = \prod_{i=1}^{n}\E\left(X_{j_{s_{\chi}(i)}}^{\chi\circ s_{\chi}(i)}\right).
\end{equation*}
which is equal to $\E(X_{J}^{\chi}$ by the splitting property, thus the proof is complete.
\end{proof}

The assumptions of Theorem \ref{thm:weakdefinetti} are much stronger than what one may wish for a de Finetti theorem. They should therefore be weakened by answering the two following questions :
\begin{itemize}
\item Does strong quantum bi-invariance imply the splitting property ?
\item Can strong quantum bi-invariance be deduced from a weaker invariance condition ?
\end{itemize} 
The classical and free de Finetti theorems tell us that the answer to both questions is no for finite families, so that we will from now on consider infinite families of pairs.

\subsection{Quantum bi-invariant families of pairs}\label{subsec:quantumbiinvariant}

In order to improve Theorem \ref{thm:weakdefinetti}, we will now take a more general approach. The setting is as follows : we have a noncommutative probability space $(\A, \varphi)$ and an infinite family of pairs $(x_{j}^{\ell}, x_{j}^{r})_{j\in \N}$ which is quantum bi-exchangeable in the following sense :

\begin{de}\label{de:quantumbiexchangeable}
An infinite family of pairs $(x_{j}^{\ell}, x_{j}^{r})_{j\in \N}$ is said to be quantum bi-exchangeable if for any $N\in \N$, the family of pairs $(x_{j}^{\ell}, x_{j}^{r})_{1\leqslant j\leqslant N}$ is quantum bi-exchangeable.
\end{de}

We want to build a subalgebra $B\subset M$ together with a $\varphi$-preserving conditional expectation. This subalgebra is the \emph{tail algebra}, but its definition requires some analytical structure, namely that of a W*-probability space. In this setting, one often assumes that the state $\varphi$ is faithful. However, the example of the left and right von Neumann algebras of a free group together with the canonical vector state shows that one cannot expect faithfulness in the bi-free setting. We will therefore make a weaker notion.

\begin{de}
A state $\varphi$ on a von Neumann algebra $A$ is said to be \emph{non-degenerate} if $\varphi(a^{*}xb) = 0$ for all $a, b\in A$ implies $x=0$.
\end{de}

From now on we will make the following assumptions :
\begin{itemize}
\item $\A$ is a von Neumann algebra.
\item The variables are self-adjoint, i.e. $(x_{j}^{\ell})^{*} = x_{j}^{\ell}$ and $(x_{j}^{r})^{*} = x_{j}^{r}$ for all $j$.
\item The restriction of $\varphi$ to the von Neumann algebra $M$ generated by all the elements $x_{j}^{\chi}$ is a non-degenerate normal state.
\end{itemize}
With this we can define our main object, namely the tail algebra.

\begin{de}
The tail algebra of an infinite sequence of pairs $(x_{j}^{\ell}, x_{j}^{r})_{j\in \N}$ is the von Neumann algebra
\begin{equation*}
B = \bigcap_{n\geqslant 0}W^{*}\left(\{x_{j}^{\ell}, x_{j}^{r}, j\geqslant n\}\right) \subset \A,
\end{equation*}
where $W^{*}(S)$ denotes the von Neumann algebra generated by the set $S$ in $\A$.
\end{de}

Let us denote by $\M$ the subalgebra of $\A$ generated by all the elements $x_{j}^{\chi}$, so that $M$ is its weak closure. Because $\varphi$ is not assumed to be faithful on $M$, we will not be able to construct a normal conditional expectation from $M$ onto $B$. We will therefore only build a map on $\M$. To do this, first note that the canonical surjection $\Phi$ from $C(S_{N}^{+})$ onto $C(S_{N})$ given at the end of Section \ref{sec:preliminaries} intertwines the linear action $\beta_{N}$ with the corresponding linear action of $S_{N}$. But since $C(S_{N})$ is commutative, the linear action is just the usual action permuting the pairs of random variables. In other words, any quantum bi-exchangeable family of pairs is classically bi-exchangeable in the following sense :

\begin{de}
A family $(x_{j}^{\ell}, x_{j}^{r})_{j\in \N}$ of pairs of random variables is \emph{classically bi-exchangeable} (or simply bi-exchangeable) if for any $N\in \N$ and any permutation $\sigma\in S_{N}$, $(x_{\sigma(j)}^{\ell}, x_{\sigma(j)}^{r})_{1\leqslant j\leqslant N}$ has the same joint distribution as $(x_{j}^{\ell}, x_{j}^{r})_{1\leqslant j\leqslant N}$, i.e. for any $n\in \N$ and any $1\leqslant j_{1}, \dots, j_{n}\leqslant N$,
\begin{equation*}
\varphi(x_{j_{1}}^{\chi(1)}\dots x_{j_{n}}^{\chi(n)}) = \varphi(x_{\sigma(j_{1})}^{\chi(1)}\dots x_{\sigma(j_{n})}^{\chi(n)}).
\end{equation*}
\end{de}

\begin{rem}
Let $\rho_{j}$ be the unique homomorphism from the noncommutative polynomial algebra $\C\langle X^{\ell}, X^{r}\rangle$ to $M$ such that $\rho_{j}(X^{\ell}) = x_{j}^{\ell}$ and $\rho_{j}(X^{r}) = x_{j}^{r}$. Then, classical bi-exchangeability is equivalent to the exchangeability of the sequence $(\rho_{j})_{j\in \N}$ in the sense of \cite[Def 4.3]{curran2009quantum}.
\end{rem}

Consider now the injective $*$-homomorphism $\gamma : M\rightarrow M$ defined by $\gamma(x_{j}^{\chi}) = x_{j+1}^{\chi}$ for any $\chi\in \{\ell, r\}$. By \cite[Thm 5.1.10]{dykema2014quantum}, the equation
\begin{equation*}
\EE(x) = \lim_{n} \gamma^{n}(x),
\end{equation*}
where the limit is in the weak-$*$ sense, defines a linear map $\EE : \M\rightarrow M$. Let us denote by $\NN$ the algebra generated by $B$ and all the variables $(x_{j}^{\chi})_{j\in \N, \chi\in \{\ell, r\}}$. It was proved in \cite[Thm 5.1.10]{dykema2014quantum} that $\EE$ is a $\varphi$-preserving norm-one projection from $\NN$ onto $B$. In order to link this construction to Theorem \ref{thm:weakdefinetti}, we will prove by adapting \cite[Prop 6.4]{dykema2014quantum}, that the pairs also satisfy a form of strong quantum bi-invariance.

\begin{prop}\label{prop:invariantexpectation}
The joint $B$-valued distribution of $(x_{j}^{\ell}, x_{j}^{r})_{j\in \N}$ is strongly quantum bi-invariant in the sense that for any $(j_{1}, \dots, j_{n})$, any $\chi\in \{\ell, r\}^{n}$ and any $b_{1}, \dots, b_{n}\in B$,
\begin{equation}\label{eq:invariantexp}
\EE(b_{1}x_{j_{1}}^{\chi(1)}b_{2}x_{j_{2}}^{\chi(2)}\dots b_{n}x_{j_{n}}^{\chi(n)}b_{n+1})\otimes 1_{C(S_{N}^{+})} = \sum_{I}^{N}\EE(b_{1}x_{i_{1}}^{\chi(1)}b_{2}x_{i_{2}}^{\chi(2)}\dots b_{n}x_{i_{n}}^{\chi(n)}b_{n+1})\otimes u^{\chi}_{IJ}.
\end{equation} 
\end{prop}

\begin{proof}
Using the strategy explained in the proof of \cite[Prop 6.4]{dykema2014quantum}, it is enough to prove the corresponding statement with $\EE$ replaced by $\varphi$. The proof then follows the argument of \cite[Lem 6.2]{dykema2014quantum} except that we have to take care of the left-right colourings all along. Approximating the elements of $B$ by Kaplansky's density theorem, we only have to prove the result for
\begin{equation*}
b_{k} = x_{p_{k}(1)}^{\chi_{k}(1)}\dots x_{p_{k}(m(k))}^{\chi_{k}(m(k))}.
\end{equation*}
Let $N$ be an integer such that $1\leqslant j_{1}, \dots, j_{n}\leqslant N$. By classical bi-exchangeability, we can assume that $M\geqslant p_{k}(i) \geqslant N+1$ for all $k$ and $i$. Let us rename the indices so that
\begin{equation*}
b_{1}x_{j_{1}}^{\chi(1)}b_{2}x_{j_{2}}^{\chi(2)}\dots b_{n}x_{j_{n}}^{\chi(n)}b_{n+1} = x_{t_{1}}^{\chi'(1)}\dots x_{t_{d}}^{\chi'(d)},
\end{equation*}
where $d = m(1)+\dots +m(n+1) + n+1$. Let $u_{N}$ be the generating matrix of $S_{N}$ and set $\widetilde{u} = u_{N}\oplus \Id_{M}$. Applying quantum bi-exchangeability with respect to $\widetilde{u}$ yields
\begin{equation*}
\varphi(x_{t_{1}}^{\chi'(1)}\dots x_{t_{d}}^{\chi'(d)}) = \sum_{q_{1}, \dots, q_{d} = 1}^{M}\varphi(x_{q_{1}}^{\chi'(1)}\dots x_{q_{d}}^{\chi'(d)})\widetilde{u}^{\chi'}_{TQ}.
\end{equation*}
Now, if $t_{s_{\chi'}(i)}, q_{s_{\chi'(i)}} \geqslant N+1$, then the corresponding coefficient of $\widetilde{u}$ is $0$ unless these two indices are equal, in which case it yields $1$. Thus, we can remove many indices in the sum. More precisely, set $a_{k} = m(1)+\dots +m(k)+k$ and note that $t_{a_{k}} = j_{k}$. Setting $\{c_{1} < \dots < c_{n}\} = s_{\chi'}^{-1}(\{a_{1}, \dots, a_{n}\})$ we have
\begin{equation*}
\varphi(x_{t_{1}}^{\chi'(1)}\dots x_{t_{d}}^{\chi'(d)}) = \sum_{q_{c_{1}}, \dots, q_{c_{n}} = 1}^{N}\varphi(b_{1}x_{c_{1}}^{\chi'(c_{1})}b_{2}\dots b_{n}x_{c_{n}}^{\chi'(c_{n})}b_{n+1})u^{\chi'\circ s_{\chi'}(c_{1})}_{t_{s_{\chi'}(c_{1})}q_{s_{\chi'}(c_{1})}}\dots u^{\chi'\circ s_{\chi'}(c_{n})}_{t_{s_{\chi'}(c_{n})}q_{s_{\chi'}(c_{n})}}.
\end{equation*}
Noticing that $\chi'\circ s_{\chi'}(c_{k}) = \chi'(a_{k}) = \chi(k)$ by definition, we see that $s_{\chi'}(c_{k}) = c_{s_{\chi}(k)}$, hence the result.
\end{proof}

It is now natural to look also for the splitting property. In fact, $\EE$ satisfies a \emph{factorization property} by \cite[Prop 6.5]{dykema2014quantum} :

\begin{prop}\label{prop:factorization}
Let $(x_{j}^{\ell}, x_{j}^{r})_{j\in \N}$ be an infinite sequence which is classically bi-exchangeable. Then, for any $n\in \N$, any $j_{1}, \dots, j_{n}$ such that $j_{k}$ is different from all other $j_{i}$'s, and any $X_{j_{i}}^{\chi(i)}\in C_{j_{i}}^{\chi(i)}$.
\begin{equation*}
\EE\left(X_{j_{1}}^{\chi(1)}\dots X_{j_{n}}^{\chi(n)}\right) = \EE\left(X_{j_{1}}^{\chi(1)}\dots \EE(X_{j_{k}}^{\chi(k)})\dots X_{j_{n}}^{\chi(n)}\right).
\end{equation*}
\end{prop}

We can summarize our results to compare them to Theorem \ref{thm:weakdefinetti}.

\begin{thm}\label{thm:incompletedefinetti}
Let $(x_{j}^{\ell}, x_{j}^{r})_{j\in \N}$ be an infinite family of pairs of random variables which is quantum bi-invariant and let $B$ be their tail algebra. Then, there exists a (non-normal) conditional expectation onto $B$ with respect to which the pairs are strongly quantum bi-invariant and such that for any $n$ different indices $j_{1}, \dots, j_{n}$ and $X_{j_{i}}^{\chi(i)}\in C_{j_{i}}^{\chi(i)}$,
\begin{equation*}
\EE\left(X_{j_{1}}^{\chi(1)}\dots X_{j_{n}}^{\chi(n)}\right) = \prod_{i=1}^{n}\EE\left(X_{j_{i}}^{\chi(i)}\right).
\end{equation*}
\end{thm}

\begin{proof}
Just apply $n$ times Proposition \ref{prop:factorization}.
\end{proof}

Comparing Theorem \ref{thm:incompletedefinetti} and Theorem \ref{thm:weakdefinetti}, we see that the only difference is that when computing expectation of a product of elements coming from $n$ different pairs, the individual expectations are not multiplied in the same. The point is that the general setting lacks a crucial feature of bi-freeness : we have only one copy of $B$. To obtain the splitting property, one would need two commuting copies of $B$, one opposite to the other. Let us illustrate this in the bi-free setting. Consider a $B$-$B$-noncommutative probability space $(\A, \E, \varepsilon)$ where $\A$ is a von Neumann algebra equipped with a nondegenerate normal state $\varphi$ which is compatible with $\E$. If a sequence of $B$-pairs in $\A$ is classically bi-exchangeable, then for any polynomial $p$ with coefficients in $B$ (where the coefficients do not commute with the unknown variable), the sequence $p(x_{N}^{\chi})$ converges weakly by \cite[Thm 5.1.10]{dykema2014quantum}. This gives a partial answer to the first question asked at the end of Subsection \ref{subsec:weakdefinetti}.

\begin{lem}\label{lem:towardssplitting}
Let $(x_{j}^{\ell}, x_{j}^{r})_{j\in \N}$ be a quantum bi-exchangeable family of $B$-pairs and assume that
\begin{equation*}
\lim p(x_{N}^{\ell}) = L_{\E(p(x_{1}^{\ell}))} \text{ and } \lim p(x_{N}^{r}) = R_{\E(p(x_{1}^{r}))}.
\end{equation*}
Then, the family of pairs has the splitting property.
\end{lem}

\begin{proof}
It is proved in \cite{dykema2014quantum} that for $N$ big enough, we can replace $p(x_{N}^{\chi})$ by its limit in $\E(ap(x_{N}^{\chi})b)$ for any $a$ and $b$. Applying this property $n$ times gives the result.
\end{proof}

This illustrates one of the difficulties in dealing with bi-freeness : how can we compare a $B$-noncommutative probability space and a $B$-$B$-probability space ? There is in fact one way of doing this, which is explained in \cite[Rem 3.2.3]{charlesworth2014combinatorics} : the actions of $B$ are simply given by multiplication on the left (we will denote by $L_{b}$ such an operator) and on the right (we will denote by $R_{b}$ such an operator) on $M$. Then, $M$ can be embedded into $\LL_{\ell}(M)$ (resp. $\LL_{r}(M)$) by left (resp. right) multiplication on itself. We will denote by $\widetilde{x}_{j}^{\ell}$ the image of $x_{j}^{\ell}$ in $M\subset \LL_{\ell}(M)$ and by $\widetilde{x}_{j}^{r}$ the image of $x_{j}^{r}$ in $M^{op}\subset\LL_{r}(M)$. The pairs $(\widetilde{x}_{j}^{\ell}, \widetilde{x}_{j}^{r})_{j\in \N}$ are then $B$-pairs in the sense of Definition \ref{de:bifreeness}. We also have a new expectation map $\E : \LL(M)\rightarrow B$ obtained by first applying the element to $1_{M}$ and then applying $\EE$. Thus, $\E(x_{i_{1}}^{\chi(1)}\dots x_{i_{n}}^{\chi(n)})$ is obtained by putting all the left variables on the left, all the right variables on the right in reversed order and then applying $\EE$. The drawback that this definition of $\E$ means that the distribution of the pairs $(\widetilde{x}_{j}^{\ell}, \widetilde{x}_{j}^{r})_{j}$ only depends on some specific moments of the original pairs $(x_{j}^{\ell}, x_{j}^{r})_{j}$ (those where a right variable is never followed by a left one), hence some information is lost. For instance, one can easily see that $(\widetilde{x}_{j}^{\ell}, \widetilde{x}_{j}^{r})_{j}$ will be bi-free as soon as the original pairs are pairwise free.

\section{Beyond bi-freeness}\label{sec:nfreeness}

Understanding bi-freeness with amalgamation as invariance under all the actions $\beta_{N}$ as defined in Definition \ref{de:permutationaction}, we can identify some clues for putting bi-freeness in a more general context. Note however that this a purely quantum algebraic approach and that we do not have any concrete model.

\begin{itemize}
\item The definition of $\beta$ involves the permutations $s_{\chi}$ in order to respect the combinatorial rule for bi-freeness : "Pull all $\ell$-points to the left and all $r$-points to the right, then inverse the order of the $r$ points." In a more general framework, we could replace these permutations by different ones.
\item Replacing $\{\ell, r\}$ by an arbitrary finite set, the definition of the action $\beta_{N}$ can easily be generalized to an action on $n$-tuples $(x_{j}^{(1)}, \dots, x_{j}^{(n)})_{1\leqslant j\leqslant N}$ of random variables as soon as a suitable set of permutations $s_{\chi}$ is given. This could lead to some concept of "$n$-freeness".
\item One may try to let different quantum groups $\G_{N}^{(k)}$ act on the different entries $x_{j}^{(k)}$. For bi-freeness, this amounts to the question : can we let a quantum group $\G_{N}^{\ell}$ act on the left variables, and a different quantum group $\G_{N}^{r}$ act on the right variables in a way compatible with bi-freeness ? We will show in Proposition \ref{prop:nobifreemixing} that this is in general not possible.
\end{itemize}

We will combine the previous considerations in a very general definition of "quantum invariance of tuples". Before that, let us fix some notations. We denote by $P(k)^{\{1, \dots, n\}}$ the set of partitions on $k$ points, where each point is decorated by a number $l\in\{1, \dots, n\}$. For a partition $\pi\in P(k)^{\{1, \dots, n\}}$, we denote this decoration pattern by $\chi_{\pi} : \{1, \dots, k\}\rightarrow \{1, \dots, n\}$. From now on, we fix a noncommutative probability space $(\A, \varphi)$. For $n ,N\in\N$, we consider the following objects :
\begin{itemize}
\item A finite family $(x_{j}^{(1)}, \dots, x_{j}^{(n)})_{1\leqslant j\leqslant N}$ of $n$-tuples of \emph{self-adjoint} random variables in $\A$. We will denote by $\M\subset\A$ the algebra that they generate.
\item For every $k\in\N$, a set of permutations $\left\{s_{\chi}\in S_{k} \: \vert\: \chi : \{1, \dots, k\}\rightarrow \{1, \dots, n\}\right\}$. We denote by $\Sigma$ the union of all these sets.
\item Compact quantum groups $\G_{N}^{(1)}, \dots, \G_{N}^{(n)}\subset O_{N}^{+}$ with associated universal $C^{*}$-algebras respectively $C(\G^{(1)}), \dots, C(\G^{(n)})$ and generating matrices respectively $(u_{ij}^{(1)}), \dots, (u_{ij}^{(n)})$ for $1\leqslant i, j\leqslant N$.
\end{itemize}

As before we set, for $I=(i_{1}, \dots, i_{k})$, $J=(j_{1}, \dots, j_{k})$ and  $\chi : \{1, \dots, k\}\rightarrow \{1, \dots, n\}$,
\begin{equation*}
x_{I}^{\chi} = x_{i_{1}}^{\chi(1)}\dots x_{i_{k}}^{\chi(k)} \text{ and } u_{IJ}^{\chi, \Sigma} = u_{i_{s_{\chi}(1)}j_{s_{\chi}(1)}}^{\chi(s_{\chi}(1))}\dots u_{i_{s_{\chi}(k)}j_{s_{\chi}(k)}}^{\chi(s_{\chi}(k))}
\end{equation*}
We are now ready for our definition, which should be compared to Definition \ref{de:permutationaction}. 

\begin{de}\label{de:NInvariance}
Let $\CC$ be a collection of sets $\CC(k)\subset P(k)^{\{1, \dots, n\}}$ of partitions and let $A(\CC)$ be the quotient of the (unital) free product of all the C*-algebras $C(\G^{(k)})$ by the relations
\begin{equation*}
\sum_{\underset{\pi\leqslant\ker(I)}{I}}^{N} u_{IJ}^{\chi_{\pi}, \Sigma} =\delta_{\pi}(J).1
\end{equation*}
for all $\pi\in\CC$ and all $J$. We define a "linear action" $\beta_{N}^{\Sigma} : \M\rightarrow \M\otimes A(\CC)$ by :
\begin{equation*}
\beta_{N}^{\Sigma}(x_{J}^{\chi}) = \sum_{I}^{N} x_{I}^{\chi}\otimes u_{IJ}^{\chi, \Sigma}
\end{equation*}
We say that the family $(x_{j}^{(1)}, \dots, x_{j}^{(n)})_{1\leqslant j\leqslant N}$ is \emph{$A(\CC)$-$\Sigma$-exchangeable}, if $\varphi$ is invariant under $\beta_{N}^{\Sigma}$, i.e.
\begin{equation*}
\varphi(x_{J}^{\chi}).1_{A(\CC)} = \sum_{I}^{N}\varphi(x_{I}^{\chi})u_{IJ}^{\chi, \Sigma}
\end{equation*}
as an equality in $A(\CC)$. An infinite family $(x_{j}^{(1)}, \dots, x_{j}^{(n)})_{ j\in\N}$ is \emph{$A(\CC)$-$\Sigma$-exchangeable}, if $(x_{j}^{(1)}, \dots, x_{j}^{(n)})_{1\leqslant j\leqslant N}$ is $A(\CC)$-$\Sigma$-exchangeable for any $N\in\N$.
\end{de}

Let us emphasize the fact that $A(\CC)$ need not be the universal C*-algebra of a compact quantum group anymore.

It seems like one can make arbitrary combinations of compact quantum groups in Definition \ref{de:NInvariance}. However, requiring some freeness for the tuples may force some of the quantum groups to be identical. We will illustrate this in the case of bi-freeness under an additional assumption. Let us first give a lemma.

\begin{lem}\label{lem:nobifreemixing}
If $\CC$ in the above definition contains the partition $\pi=\{\{1, 2\}\}$ with decoration $\chi = (k_{1}, k_{2})$, then $u_{ij}^{(k_{1})} = u_{ij}^{(k_{2})}$ in $A(\CC)$. In other words, the quantum groups $\G_{N}^{(k_{1})}$ and $\G_{N}^{(k_{2})}$ are identified in $A(\CC)$ if the partition connecting one $k_{1}$-point with one $k_{2}$-point is in $\CC$.
\end{lem}

\begin{proof}
If $s_{\chi}$ is the identity for $\chi = (k_{1}, k_{2})$, then the partition $\pi$ yields the following relation in $A(\CC)$ :
\begin{equation*}
\sum_{k} u_{ki}^{(k_{1})} u_{kj}^{(k_{2})} = \delta_{ij}
\end{equation*}
If $s_{\chi}$ is the transposition, we simply swap $k_{1}$ and $k_{2}$. Now, since $\G^{(k_{1})}_{N}$ is a quantum subgroup of $O_{N}^{+}$, we have the following relation in $A_{\G^{(1)}}(N)$ (and hence also in $A(\CC)$) :
\begin{equation*}
\sum_l u_{il}^{(k_1)} u_{jl}^{(k_1)}=\delta_{ij}
\end{equation*}
Putting things together, we get
\begin{equation*}
u_{ij}^{(k_{1})} = \sum_{l} u_{il}^{(k_{1})}\delta_{lj} = \sum_{k, l} u_{il}^{(k_{1})}u_{kl}^{(k_{1})}u_{kj}^{(k_{2})} = \sum_{k}\delta_{ik}u_{kj}^{(k_{2})} = u_{ij}^{(k_{2})}.
\end{equation*}
\end{proof}

\begin{prop}\label{prop:nobifreemixing}
Let $(\A, \E, \varepsilon)$ be a $B$-$B$-noncommutative probability space and let $(x_{j}^{\ell}, x_{j}^{r})_{j\in\N}$ be a family of $B$-pairs which are bi-free and identically distributed with amalgamation over $B$. Assume furthermore that there are compact quantum groups $\G_{N}^{\ell}$ and $\G_{N}^{r}$ such that the family of pairs is $A(\CC)$-$\Sigma$-exchangeable for a set of partitions $\CC$. If
\begin{equation*}
\varphi(x_{j}^{\ell}x_{j}^{r}) \neq \varphi(x_{j}^{\ell})\varphi(x_{j}^{r}),
\end{equation*}
then $u_{ij}^{\ell}=u_{ij}^{r}$ in $A(\CC)$.
\end{prop}

\begin{proof}
Set $y_{j}^{\ell} = x_{j}^{\ell} - \E(x_{j}^{\ell}) \text{ and } y_{j}^{r} = x_{j}^{r} - \E(x_{j}^{r})$ and let $N\in\mathbb{N}$. We are going to prove that the relations of $\pi = \{\{1, 2\}\}$ with decoration $\chi = (\ell, r)$ are fulfilled for $u_{ij}^{\ell}$ and $u_{ij}^{r}$ from which, by Lemma \ref{lem:nobifreemixing}, the assertion will follow. Set $\alpha =\E(y_{j}^{\ell}y_{j}^{r})$ which is independent from $j$ by the identical distribution assumption. Note that $\E(y_{j_{1}}^{\ell}y_{j_{2}}^{r}) = \delta_{j_{1}j_{2}}\alpha$ since $\E(y_{j_{1}}^{\ell}y_{j_{2}}^{r}) = \E(y_{j_{1}}^{\ell})\E(y_{j_{2}}^{r}) = 0$ for $j_{1}\neq j_{2}$ by independence. It is easy to check that also $(y_{j}^{\ell}, y_{j}^{r})_{j}$ is invariant under the $(\G_{N}^{\ell}, \G_{N}^{r})$-action, hence
\begin{eqnarray*}
\delta_{j_{1}j_{2}}\alpha\otimes 1 & = & \E(y_{j_{1}}^{\ell}y_{j_{2}}^{r})\otimes 1 \\
& = & \sum_{i_{1}, i_{2}}\E(y_{i_{1}}^{\ell}y_{i_{2}}^{r})\otimes u_{i_{1}j_{1}}^{\ell}u_{i_{2}j_{2}}^{r} \\
& = & \sum_{i_{1}\neq i_{2}}\E(y_{i_{1}}^{\ell}y_{i_{2}}^{r})\otimes u_{i_{1}j_{1}}^{\ell}u_{i_{2}j_{2}}^{r} + \sum_{m}\E(y_{m}^{\ell}y_{m}^{r})\otimes u_{mj_{1}}^{\ell}u_{mj_{2}}^{r}\\
& = & \alpha\otimes\sum_{m} u_{mj_{1}}^{\ell}u_{mj_{2}}^{r}
\end{eqnarray*}
Since $\E$ is $\varphi$-invariant, $\alpha$ must be non-zero so that we can conclude by Lemma \ref{lem:nobifreemixing}. 
\end{proof}

Here are some examples of $A(\CC)$-$\Sigma$-invariance.

\begin{prop}
Let $(\mathcal A,\varphi)$ be a noncommutative probability space.
\begin{itemize}
\item[(a)] Let $(x_{j})_{j\in\N}$ be a family of self-adjoint random variables in $\A$. Let $s_{\chi} = e$ (the identity permutation in $S_{k}$) for all $\chi$, let $\G_{N}^{(1)} = S_{N}^{+}$ for all $N$ and let $\CC$ be empty. Then, $(x_{j})_{j\in\N}$ is $A(\CC)$-$\Sigma$-exchangeable if and only if it is quantum exchangeable (in the sense of \cite{kostler2009noncommutative}) if and only if $(x_{j})_{j\in\N}$ is free and identically distributed with amalgamation over the tail algebra.
\item[(b)] Let $(x_{j}^{\ell}, x_{j}^{r})_{j\in\N}$ be a family of pairs of random variables in $\A$. Let $s_{\chi}$ be defined as in Definition \ref{de:binoncrossing}, let $\G_{N}^{(1)} = \G_{N}^{(2)} = S_{N}^{+}$ for all $N$ and let $\CC$ be given by all bi-noncrossing partitions. Then, $(x_{j}^{\ell}, x_j^{r})_{j\in\N}$ is $A(\CC)$-$\Sigma$-exchangeable if and only if it is quantum bi-exchangeable (in the sense of Definition \ref{de:quantumbiexchangeable}).
\item[(c)] Let $(x_{j}^{(1)}, \dots, x_{j}^{(n)})_{j\in\N}$ be a family of self-adjoint random variables in $\A$. Let $s_{\chi} = e$ (the identity permutation in $S_{k}$) for all $\chi$, let $\G_{N}^{(k)} = S_{N}^{+}$ for all $N$ and all $k$ and let $\CC$ be empty. Then, $(x_{j}^{(1)}, \dots, x_{j}^{(n)})_{j\in\N}$ is $A(\CC)$-$\Sigma$-exchangeable if and only if all the random variables $(x_{j}^{(k)})_{j\in \N, 1\leqslant k\leqslant n}$ are free with amalgamation over the tail algebra and their distribution only depends on $k$.
\end{itemize}
\end{prop}

\begin{proof}
(a) The action $\beta_{N}^{\Sigma}$ boils down to the usual action of $S_{N}^{+}$ as considered in \cite{kostler2009noncommutative} (note that for this free de Finetti theorem, in fact only a linear action is needed).

(b) By Lemma \ref{lem:nobifreemixing}, the two copies of $S_{N}^{+}$ in $A(\CC)$ coincide and $A(\CC)$-$\Sigma$-exchangeability yields exactly quantum bi-exchangeability.

(c) Since $\CC = \emptyset$, $A(\CC)$ is the free product $C(\G^{(1)})\ast \dots\ast C(\G^{(n)})$. Fixing a $k\in\{1, \dots, n\}$, the action $\beta_{N}^{\Sigma}$ restricted to variables $(x_{j}^{(k)})_{j\in\N}$ amounts to the same actions as in (a). This proves that for a fixed $k$, the variables are free and identically distributed over some "$k$-tail algebra" $B^{(k)}$, hence, 
they are free and identically distributed with amalgamation over the tail algebra $B$. Let $A_{j}$ be the algebra generated by $x_{j}^{(k)}$ for all $k$ and let $\rho_{j} : \C\langle X_{1}, \dots, X_{n}, X_{1}^{*}, \dots, X_{n}^{*}\rangle \rightarrow A_{j}$ be the unique $*$-homomorphism satisfying $\rho_{j}(X_{k}) = x_{j}^{(k)}$. Considering the quotient of $A(\CC)$ obtained by identifying all the copies of $S_{N}^{+}$, we see that $\varphi$ is invariant under quantum permutations of the tuples. This is equivalent to saying that $(\rho_{j})_{j}$ is a quantum exchangeable sequence, hence by \cite[Thm 1.1]{curran2009quantum} the algebras $A_{j}$ are free and identically distributed with amalgamation over the tail algebra $B$. On the other hand, assume that $(x_{j}^{(k)})_{j\in \N, 1\leqslant k\leqslant n}$ is free with amalgamation over the tail algebra and that the distribution only depends on $k$. Using all these properties, we can compute, for any $I$, $J$ and $\chi$,
\begin{eqnarray*}
\sum_{I}^{N}\varphi(x_{I}^{\chi})u_{IJ}^{\chi, \Sigma} & = & \sum_{I}^{N}\sum_{\pi\in NC}\kappa_{\pi}(x_{I}^{\chi})u_{IJ}^{\chi, \Sigma} \\
& = & \sum_{\underset{\pi\leqslant \ker(\chi)}{\pi\in NC}}\sum_{\underset{\pi\leqslant \ker(I)}{I}}^{N}\kappa_{\pi}(x_{I}^{\chi})u_{IJ}^{\chi, \Sigma} \\
& = & \sum_{\underset{\pi\leqslant \ker(\chi)}{\pi\in NC}}\kappa_{\pi}^{\chi}\sum_{\underset{\pi\leqslant \ker(I)}{I}}^{N}u_{IJ}^{\chi, \Sigma}
\end{eqnarray*}
where $\kappa_{\pi}$ denotes the usual noncrossing cumulants. Because $\Sigma$ only contains identity permutations, we have
\begin{equation*}
\sum_{\underset{\pi\leqslant \ker(I)}{I}}^{N}u_{IJ}^{\chi, \Sigma} = \sum_{\underset{\pi\leqslant \ker(I)}{I}}^{N}u_{i_{1}j_{1}}^{\chi(1)}\dots u_{i_{n}j_{n}}^{\chi(n)}.
\end{equation*}
Since $\pi\leqslant \ker(\chi)$, the latter sum is equal to $\delta_{\pi\leqslant \ker(J)}$ by definition of the free product. Writing $\varphi(x_{J}^{\chi})$ as a sum of cumulants, we then see that $\varphi$ is invariant.
\end{proof}

Based on these constructions and remarks, we see that the combinatorics of $n$-freeness should be governed by the set $NC^{\Sigma}(\chi)$ of partitions $\pi\in P^{\{1, \dots, n\}}$ such that $s_{\chi}^{-1}(\pi)$ is noncrossing. Given $\Sigma$, one can define expectation functionals and, by Möbius inversion, cumulants satisfying the expected identities at least in the scalar-valued case. However, going to the operator-valued setting requires understanding how elements of the $n$ copies of the algebra $B$ behave with respect to these functionals, which is quite unclear for the moment. The best way to understand these properties is to have an operator model and it may be that one has to impose conditions on the set of permutations $\Sigma$ for such an operator model to exist, even when $B = \C$.

\bibliographystyle{amsplain}
\bibliography{../../quantum}

\end{document}